\documentclass{elsarticle}%
\usepackage{amsfonts}
\usepackage{amsthm}
\newtheorem{openproblem}{Open Problem}[section]
\usepackage{amsmath}
\usepackage{amssymb}
\usepackage{bbold}
\usepackage{mathtools}
\usepackage{xcolor}
\usepackage{soul}
\usepackage{mathrsfs}
\usepackage{tabularx}
\usepackage[top=1in, bottom=1in, left=1in, right=1in]{geometry}
\usepackage{graphicx}%
\providecommand{\U}[1]{\protect\rule{.1in}{.1in}}
\newtheorem{theorem}{Theorem}[section]
\newtheorem{lemma}{Lemma}[section]
\newtheorem{corollary}{Corollary}[section]

\theoremstyle{definition}
\newtheorem{definition}{Definition}[section]
\theoremstyle{remark}

\numberwithin{equation}{section}
\begin{document}
    \begin{frontmatter}

\title{On the Diophantine Equation  $F_n = F_l^k (F_l^m - 1)$
}

%%this line removes the date, but space is still left for it;
%if used, remove the \vspace{-1cm}
%\date{}

%this gives the date in the form Mon 30 Jan 2012, 8:57pm;
%if used, retain the \vspace{-1cm}
%\date{\shortdayofweekname{\day}{\month}{\year}{ }\mydate\today}

\author[1]{Seyran S. Ibrahimov}
\ead{seyran.ibrahimov@emu.edu.tr}
\author[1,2]{Nazim I. Mahmudov}
\ead{nazim.mahmudov@emu.edu.tr}

\address [1] {Department of Mathematics, Eastern Mediterranean University, Mersin 10, 99628, T.R. North Cyprus, Turkey}
\address [2] {Research Center of Econophysics, Azerbaijan State University of Economics (UNEC), Istiqlaliyyat Str. 6, Baku 1001,Azerbaijan}

% Latex won't make the title unless told:
%\maketitle

%%to remove the space left for date, use:

\begin{abstract}
The sequence \(\{ F_t \}_{t=0}^{\infty}\) represents the Fibonacci numbers, defined by the recurrence relation \(F_0 = 0\), \(F_1 = 1\), and \(F_t = F_{t-1} + F_{t-2}\) for \(t \geq 2\). In this paper, we prove that \((n, l, k, m) = (6, 3, 3, 1)\) is the only solution to the Diophantine equation \(F_n = F_l^k (F_l^m - 1)\), where \(n, l, m \geq 1\) and \(k \geq 3\). To solve this problem, we apply Matveev's theorem, which provides lower bounds for linear forms in logarithms of algebraic numbers, in combination with a modified Baker-Davenport reduction technique and a divisibility property of Fibonacci numbers.
\end{abstract}


\end{frontmatter}

\section{INTRODUCTION}
In this research, we analyze the Diophantine equation
 \begin{equation}\label{1.1}
F_n = F_l^k (F_l^m - 1)
\end{equation}
where $n$, $l$, $m \geq 1$ and $k \geq 3$.
First, we review some related Diophantine problems that have been explored in the literature. For instance, by selecting \( m = 1 \) and \( l = 3 \) in equation \eqref{1.1}, we get \( F_n = 2^k \), while choosing \( k = 1 \) and \( l = 3 \) results in \( F_n = 2^{m+1} - 2 \), both of which are particular cases of Theorem 1 in Ddamulira et al.'s work in \cite{7}. Additionally, by setting $l = 3$,  $k = p - 1$, and $ m = p$ in the Diophantine equation (\ref{1.1}), and assuming that both $p$ and $2^{p} - 1$ are prime numbers, we arrive at the following equation:

\begin{equation}\label{1.2}
F_{n}=2^{p-1}(2^{p}-1)
\end{equation}
The solutions to equation (1.2) lead to even perfect numbers within the Fibonacci sequence. However, Luca, in \cite{12}, proved that there are no perfect numbers in the Fibonacci sequence.
Moreover, Facó and Marques \cite{10} explored a version of equation \eqref{1.2} in which the left-hand side involves generalized Fibonacci numbers. 
Generally, in recent decades, mathematicians have intensively explored exponential Diophantine equations that involve terms from second-order linear recurrence sequences. For instance, in \cite{4}, Bravo and Luca investigated the Diophantine equation \( F_n + F_m = 2^a \), which generalizes the problem \( F_n = 2^k \). Moreover, Luca and Patel, in \cite{13}, demonstrated that all solutions of the Diophantine equation $F_n \pm F_m = y^p$ in integers $(n, m, y, p)$ with $p \geq 2$ and $n \equiv m \pmod{2}$ either satisfy $\max \{|n|, |m|\} \leq 36$, or $y = 0$ and $|n| = |m|$. Nevertheless, the problem remains unsolved for the case where \( n \not\equiv m \pmod{2} \).

 Research has also explored the equation $F_n - F_m = y^a$ when $y$ is fixed. Namely, Şiar and Keskin \cite{18} identified all solutions for $y = 2$, Bitim and Keskin \cite{2} determined all solutions for $y = 3$, and Erduvan and Keskin \cite{9} found all solutions for $y = 5$. In the same study \cite{9}, the authors conjectured that there are no solutions to the equation $F_n - F_m = y^a$ when $y$ is a prime number $>7$.
 Also, Luca and Szalay \cite{15} and Luca and  Stănică \cite{14} demonstrated that the Diophantine equations $F_n = p^a \pm p^b + 1$ and $F_n = p^a \pm p^b$, respectively, each have only a finite number of positive integer solutions \((n, p, a, b)\) where $n \geq 3$ and $a \geq \max\{2, b\}$, with $p$ being an indeterminate prime number.
 
In the proof of the main result, we use Matveev's theorem to obtain an upper bound for the variables $n$ and $m$ in the Diophantine equation (1.1). Then, to simplify these bounds, we use a modified form of Baker and Davenport's reduction lemma, along with a specific divisibility property of Fibonacci numbers, which helps us derive bounds for $l$ and $k$. We also use Python for some of the calculations in this article.

\section{Preliminary Lemmas}

\begin{lemma}\cite{11}
The nth term of the Fibonacci sequence can be found using the Binet formula as shown below:
\begin{align*}
F_{n}=\frac{\phi^{n}-(-\phi)^{-n} }{\sqrt5}, \quad n\geq0
\end{align*}
where
\begin{align*}
\phi=\frac{1+\sqrt5}{2}.
\end{align*}
\end{lemma}

By applying Lemma 2.1, we can directly derive the following inequalities.

\begin{corollary}\cite{11}\label{Golden ratio ineq}
\begin{align*}
\phi^{n-2}\leq F_{n} \leq \phi^{n-1},\quad n\geq1.
\end{align*}
\end{corollary}

We will now articulate a specific property concerning the divisibility of Fibonacci numbers.
In this context, for integers $a \geq 2$, $k \geq 0$, and $b \geq 1$, we denote that $a^k$ exactly divides $b$ by writing $a \parallel b$ if $a^k \mid b$ and $ a^{k+1} \nmid b$.

\begin{lemma} \cite{17}
Let \( k, l, n \) be positive integers and \( k \geq 2 \) and \( l \geq 3\) . Then the following statements hold:

\begin{enumerate}
  \item \textbf{} If \( F_l^k \mid\mid F_n \) and \( l \not\equiv 3 \pmod{6} \), then \( F_l^{k-1} \mid\mid \frac{n}{l} \).
  \item \textbf{} If \( F_l^k \mid\mid F_n \) and \( l \equiv 3 \pmod{6} \), and \( 2^{k-1} \mid \frac{n}{l} \), then \( F_l^{k-1} \mid\mid \frac{n}{l} \).
  \item \textbf{} If \( F_l^k \mid\mid F_n \) and \( l \equiv 3 \pmod{6} \), and \( 2^{k-1} \nmid \frac{n}{l} \), then \( F_l^{k-2} \mid\mid \frac{n}{l} \).
\end{enumerate}
\end{lemma}
From equation (1.1), we see that \( F_l^k \mid\mid F_n \). Applying Lemma 2.2, we find that for all \( k \geq 2 \) and \( l \geq 3 \), \( F_l^{k-2} \mid \frac{n}{l} \). This leads to the conclusion that:

\begin{equation*}
l\cdot F_l^{k-2}\leq n
\end{equation*}
Hence
\begin{equation}
\log(l)+(k-2)(l-2)\log(\phi)\leq \log(n)
\end{equation}

We will now highlight several fundamental concepts from algebraic number theory.\\
Let $y$ be an algebraic number of degree $d$ and let

\begin{align*}
a_{0}\prod_{i=1}^{d}(X-y^{(i)})\in \mathbb{Z}[X]
\end{align*}
 be the minimal polynomial of $y$.
where $a_{0}> 0$ and $y^{(i)}$, $i=1,2,\dots,d$,  are the conjugates of $y$.

\begin{definition}
The logarithmic height of $y$ is defined by
\begin{align*}
h(y)=\frac{1}{d}\bigg(\log a_{0}+\sum_{i=1}^{d} \log\big(\max\lbrace \vert y^{(i)}\vert,1\rbrace\big)\bigg)
\end{align*}
\end{definition}
We will subsequently present a consequence of Matveev's theorem (\cite{5}, \cite{16}).
\begin{lemma}
Assume that $\beta_{1},\dots, \beta_{n}$ are positive algebraic numbers in a real algebraic number field $\mathbb{L}$ of degree $D$, $r_{1},\dots, r_{n}$ are rational integers, and
\begin{align*}
\Lambda := \beta_{1}^{r_{1}}\dots \beta_{n}^{r_{n}}-1\not=0.
\end{align*}
then
\begin{align} \label{Matveev ineq}
\vert \Lambda\vert > \exp \bigg(-1.4 \cdot 30^{n+3} \cdot n^{4.5} \cdot D^{2}(1+\log D)(1+\log T)A_{1}\dots A_{n}\bigg),
\end{align}
where $T\geq \max\lbrace \vert r_{1}\vert ,\dots, \vert r_{n}\vert\rbrace$, and $A_{j}\geq \max\lbrace Dh(\beta_{j}), \vert \log \beta_{j}\vert, 0.16\rbrace$, for all $j=1,\dots,n$.
\end{lemma}

Dujella and Pethő \cite{8} proposed a modification of a lemma initially introduced by Baker and Davenport \cite{1}. Later, in \cite{3}, the authors offered an alternative version of the conclusion derived from Dujella and Pethő’s lemma.
Let \( \Vert z \Vert \) denote the distance from a real number \( z \) to the nearest integer, defined as \( \Vert z \Vert = \min \{ \lvert z - n \rvert : n \in \mathbb{Z} \} \).

\begin{lemma}\label{Dujella}
Let $M$ be a positive integer, $\frac{p}{q}$ be a convergent of the continued fraction of the irrational $\gamma$ such that $q>6M$, and let $A,B,\mu$ be real numbers with $A>0$ and $B>1$. If $\varepsilon=\Vert \mu q \Vert-M\Vert \gamma q\Vert>0$, hence there is no solution to the inequality
\begin{align*}
0<\vert u\gamma-v+\mu\vert<AB^{-\omega},
\end{align*}
in positive integers $u,v$ and $\omega$ with
\begin{align*}
u\leq M\quad and \quad \omega\geq \frac{\log(\frac{Aq}{\varepsilon})}{\log B}.
\end{align*}
\end{lemma}
\section{ MAIN RESULT}
\begin{theorem} The only positive integer solution to the equation (\ref{1.1}) with $ n, l, m \geq 1$ and $k \geq 3$ is $(n, l, k, m) = (6, 3, 3, 1)$.

\end{theorem}
\begin{proof} Initially, let us explore some particular cases. For example, if we take $l = 1$ or $l = 2$ in the Diophantine equation \eqref{1.1}, we obtain $F_n = 0$, which implies that $n = 0$. So, let us assume that \( l \geq 3 \). If we set $m = 1$ in equation \eqref{1.1}, we obtain the following equation:
\begin{equation}
F_n = F_l^k (F_l - 1)
\end{equation}

It is easy to see that $F_l - 1$ divides $F_{l-2} F_{l-1} F_{l+1} F_{l+2}$. This follows from Catalan's formula \cite{11}, which states that
$F_{l-d}F_{l+d} - F_l^2 = (-1)^{l+d+1}F_d^2$,
where $d$ is a positive integer and $l \geq d$. Thus, in equation (3.1), every prime factor of $F_n$ is a prime factor of $F_l$ or one of $F_{l-2}, F_{l-1}, F_{l+1}, F_{l+2}$. Additionally, $l$ divides $n$ and $l < n$, so $l \leq \frac{n}{2}$.
By Carmichael's primitive divisor theorem \cite{6,19}, if $n \neq 1, 2, 6, 12$, then $F_n$ has at least one primitive prime factor. This means that there is a prime factor of $F_n$ that is not a prime factor of $F_s$ for any positive integer $s < n$. In particular, if $n > 12$, then $l + 2 \leq \frac{n}{2} + 2 < n$, so $F_n$ will have a prime factor that is not a divisor of any of $F_{l-2}, F_{l-1}, F_l, F_{l+1}, F_{l+2}$, which shows that there is no solution to equation (3.1). Finally, it remains to check if equation (3.1) has any solutions for $n \leq 12$. Through straightforward calculations, we determined that equation (3.1) has only two solutions, namely $(n, l, k) = \{ (3, 3, 1), (6, 3, 3) \}$, for $n \leq 12$.

We will proceed with the examination of equation \eqref{1.1} for $l \geq 3$, $m \geq 2$, and $ k \geq 3$. Under these conditions on $l$, $m$, and $k$, we get that $n \geq 9$. We now compare the two sides of equation \eqref{1.1} using Corollary \eqref{Golden ratio ineq}. We have
 \begin{align*}
     \phi^{n-2}\leq \phi^{(l-1)(k+m)}
 \end{align*}
 \begin{align*}
     \phi^{n-1}\geq \phi^{(l-2)k}(\phi^{(l-2)m}-1)
 \end{align*}
which implies that:
\begin{align}
    n\leq 2+(k+m)(l-1).
\end{align}
 
\begin{align}
    n\geq(k+m)(l-2).
\end{align}
 
Using Binet's formula, we can reformulate equation \eqref{1.1} as follows:
\begin{align*}
&\frac{\phi^{n}-(-\phi)^{-n}}{\sqrt{5}}=F_l^{k+m}-F_l^{k},\\
\end{align*}
to get
\begin{align*}
   F_l^{k+m}- \frac{\phi^n}{\sqrt{5}} &=F_l^k-\frac{(-\phi)^{-n}}{\sqrt{5}},
\end{align*}
The right-hand side above is positive. Because \bigg($F_l^k - \frac{(-\phi)^{-n}}{\sqrt{5}}\bigg) \in \left( F_l^k - \frac{1}{2\sqrt{5}}, F_l^k + \frac{1}{2\sqrt{5}} \right)$. Next, dividing both sides of last equation by $F_l^{k+m}$ we obtain
\begin{equation}
 0<1-\phi^{n}(\sqrt 5)^{-1} F_l^{-(k+m)}< F_l^{k}+\frac{1}{2{\sqrt{5}}}<\frac{1.03}{F_l^{m}}.
\end{equation}

 We first apply Lemma 2.3 to the left-hand side of inequality (3.4). We put
$$\beta_{1}=\phi, \beta_{2}=\sqrt{5}, \beta_{3}=F_l\quad and\quad r_{1}=n, r_{2}=-1, r_{3}=-(k+m).$$ 

We thus take 
\begin{align*}
    \Lambda_1 :=\phi^{n}(\sqrt 5)^{-1} F_l^{-(k+m)}-1
\end{align*}
Assuming that \(\Lambda_1 = 0\), we obtain \(\frac{\phi^n}{\sqrt{5}} = F_l^{k+m}\), which implies \(\phi^{2n} \in \mathbb{Q}\). However, from the formula \(\phi^r = \phi F_r + F_{r-1}\) for \(r \geq 0\) \cite{11}, we can easily see that \(\phi^{2n} \notin \mathbb{Q}\). Therefore, \(\Lambda_1 \neq 0\).
Given that \( T \geq \max\{n, 1, k+m\} \) and \( n \geq (k+m)(l-2) \), we can take \( T = n \). Furthermore, since \( \beta_1 \), \( \beta_2 \), and \( \beta_3 \) are elements of the real quadratic number field \( \mathbb{L} = \mathbb{Q}(\sqrt{5}) \), we choose \( D = 2 \).

Then

$$h(\beta_{1})=\frac{1}{2}\log{\phi}, h(\beta_{2})=\frac{1}{2}\log{5}, h(\beta_{3})=\log{F_l},$$

we take

$$A_{1}=\log{\phi}, A_{2}=\log{5}, A_{3}=2\log{F_l}.$$

Additionally, by combining inequalities \eqref{Matveev ineq} and (3.4), we derive:

\begin{align*}
&\frac{1.03}{F_l^{m}}>\exp \bigg(-1.4\cdot 30^{6} \cdot 3^{4.5}\cdot 2^{3} (1+\log{2})(1+\log{(n)})(\log{F_l})\log{\phi}\log{5}\bigg),\\
\end{align*}
brief calculations reveals that
\begin{align*}
&m<1.5\cdot 30^{6} \cdot 3^{4.5}\cdot 2^{3} (1+\log{2})(1+\log{(n)})\log{\phi}\log{5}
\end{align*}
Then, we achieve
\begin{equation}
m<161\cdot 10^{10}(1+\log{(n)}).
\end{equation}

Let us now reformulate equation \eqref{1.1} as follows:

\begin{align*}
&\frac{\phi^{n}-(-\phi)^{-n}}{\sqrt{5}}= F_l^k (F_l^m - 1),\\
\end{align*}
to obtain
\begin{align*}
    \frac{\phi^n}{\sqrt{5}} - F_l^k (F_l^m - 1) &= \frac{(-\phi)^{-n}}{\sqrt{5}}
\end{align*}

Taking the absolute values on both sides of the last equation and performing the necessary calculations leads to the following result:
\begin{align*}
&\bigg\vert \frac{\phi^{n}}{\sqrt{5}}-F_l^k (F_l^m - 1)\bigg\vert < \frac{1}{2{\sqrt{5}}}\\
\end{align*}
Upon dividing both sides of the inequality above by
$\frac{\phi^{n}}{\sqrt{5}}$, we obtain the following

\begin{equation}
   \vert \sqrt 5\phi^{-n} F_l^k (F_l^m - 1)-1\vert< \frac{1}{2\phi^{n}}
\end{equation}
In the following step, we apply Matveev's lemma once again to the left-hand side of inequality (3.6) with: 

$$\beta_{1}=\sqrt{5}, \beta_{2}=\phi, \beta_{3}=F_l,\beta_{3}=F_l^{m}-1\quad and\quad r_{1}=1, r_{2}=-n, r_{3}=k, r_{4}=1.$$ 
We therefore choose
\begin{align*}
    \Lambda_{2} := \sqrt{5} \phi^{-n} F_{l}^{k} (F_{l}^{m}-1)-1
\end{align*} 
Assuming \(\Lambda_{2} = 0\), we have \(\frac{\phi^{n}}{\sqrt{5}} = F_l^k (F_l^m - 1)\), which implies that \(\phi^{2n} \in \mathbb{Q}\), an impossibility. Therefore, \(\Lambda_{2} \neq 0\). Since $T \geq \max\{n, k, 1\}$ and $n \geq (k + m)(l - 2)$, we can take $T = n$. The algebraic real number field containing \(\beta_{1}\), \(\beta_{2}\), \(\beta_{3}\), and \(\beta_{4}\) is \(\mathbb{L} = \mathbb{Q}(\sqrt{5})\), which is quadratic, so we can take \(D = 2\). 

Since

$$h(\beta_{1})=\frac{1}{2}\log({5}),
h(\beta_{2})=\frac{1}{2}\log{\phi}, h(\beta_{3})=\log{F_l},
h(\beta_{4})=\log({F_l^{m}-1})$$

we take

$$A_{1}=\log({5}), A_{2}=\log{\phi}, A_{3}=2(l-1)\log{\phi}, A_{4}=2(l-1)m\log{\phi} $$

Moreover, by integrating inequalities \eqref{Matveev ineq} and (3.6), we arrive at the subsequent inequality:

\begin{align*}
&\frac{1}{2\phi^{n}}>\exp \bigg(-1.4\cdot 30^{7} \cdot 2^{13} (1+\log({2}))(1+\log{(n)})\cdot m(l-1)^2(\log{\phi})^3\log({5})\bigg),
\end{align*}
A brief calculation shows that:
\begin{equation}
n<1.4\cdot 30^{7} \cdot 2^{13} (1+\log({2}))(\log{\phi})^2\log({5})\cdot(l-1)^2(1+\log{(n)})m
  \end{equation}
 
By applying inequality (2.1), we obtain
$(l - 2) \log(\phi) < \log(n)$, and by using this bound for $l$ along with the bound for $m$ in (3.5) within inequality (3.7), we conclude that:
  \begin{equation}
n<1.4\cdot 30^{7} \cdot 2^{13}\cdot 161\cdot10^{10} (1+\log({2}))(\log{\phi})^2\log({5})\left[1+\frac{\log(n)}{\log(\phi)}\right]^2(1+\log{(n)})^2
  \end{equation}
  Revealing that
  \begin{equation}
  n<4.64\cdot10^{34}.
  \end{equation}
 
 Proceeding further, through the application of inequalities (3.5) and (3.9), we ascertain that:
 $$ m<1.61\cdot 10^{12}(1+\log(4.64\cdot10^{34})),$$
 then, we find
$m<1.31\cdot 10^{14}$.
 
 Next, using inequality (2.1) with the conditions
$k>2$ and $l>2$, we derive
\begin{equation}
  \log(l)+(l-2)\log(\phi)<\log(4.64\cdot10^{34}), \,\, \text{implies}\,\, l<158.
  \end{equation}
\begin{equation*}
  (k-2)\log(\phi)<\log(4.64\cdot10^{34}), \,\,\text{implies}\,\, k<168.
  \end{equation*}
  We will now undertake the preparation required to apply the lemmas established by Dujella and Pethő.
 
 Suppose 
 \begin{align*}\Gamma:=n\cdot\log(\phi)-(k+m)\log(F_l)-\log(\sqrt5)\end{align*}
 
Thus
 \begin{align*}
  \Lambda_1 = e^{\Gamma} -1
 \end{align*}

Then

 \begin{align*}
  \vert e^{\Gamma} -1\vert<\frac{1.03}{F_l^{m}}<\frac{1}{2},\, \text{for}\, \, m\geq 2.
 \end{align*}
 Since \( \Gamma < 0 \) and \( \lvert e^\Gamma - 1 \rvert < \frac{1}{2} \), we have \( 1 - e^\Gamma < \frac{1}{2} \), which implies \( e^{-\Gamma} < 2 \).
Here, we will use the fact that \( y < e^y - 1 \) for all \( y \neq 0 \). Then,
 
 \begin{align*}
 0<-\Gamma<e^{-\Gamma}-1=e^{-\Gamma}\vert e^{\Gamma} -1\vert<\frac{2.06}{F_l^{m}}.
  \end{align*}
 Therefore, for all $l\geq3$,
\begin{align}
-\Gamma <\frac{1.03}{2^{m-1}}.
\end{align}

 By dividing both sides of inequality (3.11) by
$\log(\phi)$, we deduce the following expression:
 \begin{align*}
0<(k+m)\frac{\log{(F_l)}}{\log{\phi}}-n+\frac{\log{\sqrt{5}}}{\log{\phi}}<\frac{1.03}{\log{\phi}} 2^{1-m}.
\end{align*}
In the subsequent phase, we will apply Lemma \ref{Dujella} by taking $u=k+m$, $\gamma_{l}=\frac{\log{(F_l)}}{\log{\phi}}$, $v=n$, $\mu = \frac{\log(\sqrt{5})}{\log(\phi)}$, $ A:=\frac{1.03}{\log{\phi}}$,\, $B=:2$ \, and \, $\omega=:m-1$ , to further refine the upper bound on $m$.

Now, let us take $M=167+1.31\cdot10^{14}$. For each of our numbers $l$ where $3\leq l\leq 157$, we take $q:=q(l) $ to be the denominator of the first convergent to $\gamma_l=\frac{\log{(F_l)}}{\log{\phi}}$  such that $q>6M$ and $\varepsilon>0$. Thus, we may aply Lemma \ref{Dujella} for each such $q$, $\gamma_l$ and $\mu$. The minimal value of $\varepsilon$ is greater than $ 1.5\cdot 10^{-28}$. Also the maximal value of $q(l)=q_{75}(154)=25431328747122828658870707509980696460342$.
Then from Lemma \ref{Dujella}, we find that

\begin{align*}
m-1<\frac{\log(\frac{Aq}{\varepsilon})}{\log B}= \frac{\log\left(\frac{1.03\cdot 25431328747122828658870707509980696460342}{\log{\phi}\cdot 1.5\cdot 10^{-28}} \right)}{\log{2}} < 226.1
\end{align*}
which implies that $m\leq 227$. From inequality (3.1) we acquire,
\begin{equation*}
n\leq 2+(k+m)(l-1)\leq 61466.
\end{equation*}
We determine that if quadruples \( (n, l, k, m) \) satisfy equation \eqref{1.1}, then the following inequalities hold true:

\begin{equation}
\begin{cases}
9\leq n\leq 61466,\\
        3\leq l\leq 157,\\
3\leq k\leq 167,\\
2\leq m\leq 227.
\end{cases}
\end{equation}
The upper bound for $n$ is quite large, making it impractical to use a computer search to find solutions to equation (1) under the conditions specified in (3.12). Therefore, we will repeat the entire process that started in (3.10) and ended in (3.12) using new, adjusted upper bounds for $n$, $l$, $k$, and $m$.

Following this, we will once more utilize inequality (2.1) to get:
\begin{equation*}
  \log(l)+(l-2)\log(\phi)< \log(61466), \,\, \text{infers}\,\, l<19.
  \end{equation*}
\begin{equation*}
  (k-2)\log(\phi)< \log(61466), \,\,\text{infers}\,\, k<25.
  \end{equation*}
 Now, we will once again use Lemma \ref{Dujella} with the same $u$, $\gamma_l$, $v$, $\mu$, $ A,$ $B$ and $\omega $. This time, we select $M = 24 + 227=251$. For each integer $l$ in the range $3 \leq l \leq 18$, we let $q := q(l)$ be the denominator of the first convergent to $\gamma_l = \frac{\log(F_l)}{\log \phi}$ such that $q > 6M$ and $\varepsilon > 0$. Consequently, we can apply Lemma \ref{Dujella} for each corresponding $q$, $\gamma_l$, and $\mu$. The minimum value of $\varepsilon$ is greater than $ 0.001274174011265825$, while the maximum value of $ q(l)$ is $q_{6}(16) = 61976$.

Based on Lemma \ref{Dujella}, we can subsequently conclude that
\begin{align*}
m-1 < \frac{\log\left(\frac{1.03 \cdot 61976}{\log\left(\phi\right) \cdot 0.001274174011265825}\right)}{\log(2)}
< 26.7
\end{align*}
Applying inequality (3.3) once more, we obtain
\begin{equation*}
n\leq 2+(k+m)(l-1)\leq 869.
\end{equation*}
We conclude that any possible solutions to equation \eqref{1.1} must fulfill the following inequalities:
\begin{equation}
\begin{cases}
9\leq n\leq 869,\\
        3\leq l \leq 18,\\
3\leq k \leq 24,\\
2\leq m\leq 27.
\end{cases}
\end{equation}

The calculations conducted with the \texttt{Python}  demonstrated that there are no positive integer solutions to equation \eqref{1.1} within the constraints specified in (3.13).

\end{proof}

\section{ CONCLUSION}
 
   In this paper, we have established that \( (n, l, k, m) = (6, 3, 3, 1) \) is the only quadruple, where \( n, l, m \geq 1 \) and \( k \geq 3 \), that satisfies the Diophantine equation \( F_n = F_l^k (F_l^m - 1) \). The cases \( k = 1 \) and \( k = 2 \) require separate investigation. However, we conjecture that, in general, \( (n, l, k, m) = \{(6, 3, 3, 1), (3, 3, 1, 1)\} \) are the only solutions to equation (1.1). Furthermore, future research could explore the following generalized form of this equation:

\begin{openproblem}
      There are only finitely many positive integer five-tuples $(n, a, k, b, m)$ that satisfy the equation $F_n = a^k(b^m - 1)$.
\end{openproblem}

\bigskip
   \textbf{Acknowledgments}
\bigskip

We would like to extend our thanks to Ali Mohammad for his support in developing the Python code required for the calculations in this paper. We also want to thank Prof. Florian Luca for the solution he proposed for equation (3.1).

\end{document}